\documentclass[12pt]{amsart}
\usepackage{comment}
\usepackage{mathrsfs}
\usepackage[margin=1in]{geometry}
\usepackage{subcaption}
\usepackage{hyperref}
\hypersetup{
    colorlinks=true,
    linkcolor=blue,
    filecolor=blue,      
    urlcolor=blue,
    citecolor=blue,
    pdftitle={Overleaf Example},
  }
  
\usepackage{amsmath}
\usepackage{amsfonts}
\usepackage{amssymb}
\usepackage{graphicx}
\usepackage{tikz-cd}
\usepackage{amsthm}
\usepackage[all]{xy}
\usepackage{adjustbox}
\usepackage{mathtools}

\newcommand{\der}[2]{\dfrac{\mathrm{d}{#1}}{\mathrm{d}{#2}}}
\newcommand{\uu}{\mathbf{u}}
\newcommand{\ww}{\mathbf{w}}
\newcommand{\vv}{\mathbf{v}}
\newcommand{\ttt}{\mathbf{t}}
\newcommand{\Pic}{\mathrm{Pic}}
\newcommand{\Cl}{\mathrm{Cl}}

\newcommand{\PP}{\mathbb{P}}
\newcommand{\ZZ}{\mathbb{Z}}
\newcommand{\RR}{\mathbb{R}}
\newcommand{\OO}{\mathcal{O}}
\newcommand{\XX}{\widetilde{X}}

\newtheorem{counter}{counter}[section]
\newtheorem{theorem}[counter]{Theorem}

\newtheorem{corollary}[counter]{Corollary}
\newtheorem{lemma}[counter]{Lemma}
\newtheorem{proposition}[counter]{Proposition}

\theoremstyle{definition}
\newtheorem{definition}[counter]{Definition}
\newtheorem{notation}[counter]{Notation}
\newtheorem{remark}[counter]{Remark}

\newtheorem{example}[counter]{Example}

\newcommand{\leqp}{%
  \mathrel{\raisebox{-0.5ex}{$\scriptscriptstyle($}}%
  \leq
  \mathrel{\raisebox{-0.5ex}{$\scriptscriptstyle)$}}%
}
\newcommand{\rk}{\mathrm{rk}}
\newcommand{\HH}{\mathrm{H}}
\newcommand{\KK}{\mathrm{K}}

\title{Spherical vector bundles on nodal $K3$ surfaces}

\author{Yeqin Liu}

\begin{document}
\maketitle

\begin{abstract}
We show that when a $K3$ surface acquires a node, the existence of stable spherical sheaves of certain Chern classes can be obstructed.
\end{abstract}

\section{Introduction}

\begin{definition}\label{nodalK3}
 In this paper, a \emph{nodal $K3$ surface} is a projective surface $X$ with a unique ordinary double point, $\omega_{X}\cong \mathcal{O}_{X}$, and $ \mathrm{H}^{1}(X, \mathcal{O}_{X})=0$.  A nodal $K3$ surface is called \emph{general}, if the group of Cartier divisors $\mathrm{Pic}(X)=\mathbb{Z}$. 
\end{definition}
Throughout the paper, let $p$ be the unique singularity of a nodal $K3$ surface $X$. The blow up $\widetilde{X}=\mathrm{bl}_{p}X$ is a smooth $K3$ surface, and the exceptional divisor $L$ is a $(-2)$-curve on $\widetilde{X}$. Let $\pi: \widetilde{X} \rightarrow X$ be the contraction. Let $\mathrm{Pic}(X)$ be the group of Cartier divisors and $\mathrm{Cl}(X)$ be the group of Weil divisors.
\subsection{Main theorem}

In \cite{Sim94}, moduli spaces of stable sheaves on singular varieties are constructed. Stable sheaves on smooth $K3$ surfaces and their moduli spaces have been extensively studied (e.g. \cite{KLS06, O'G99, PR14, Yos01, Yos99}). However, less is known when they specialize to a singular surface. In this paper we study stable spherical sheaves on general nodal $K3$ surfaces (Definition \ref{nodalK3}) using Bridgeland stability. The following main result shows that acquiring singularities can obstruct the existence of stable sheaves of certain Chern classes.

\begin{theorem}\label{main}
  Let $X$ be a nodal $K3$ surface with $\mathrm{Pic}(X)\cong \mathbb{Z}H$, and $ \vv\in \mathrm{H}^{*}_{alg}(X):= \mathrm{H}^{0}(X)\oplus \mathrm{Pic}(X)\oplus  \mathrm{H}^{4}(X)$ with $\mathbf{v}^{2}=-2$ and $\rk(\vv)>0$.
  \begin{itemize}
  	\item (Theorem \ref{nonexistence}) If $\mathrm{Cl}(X)\neq \mathrm{Pic}(X)$ and $\rk(\vv)=2$, then the moduli space of (Gieseker) $H$-semistable sheaves $M_{X, H}(\vv)$ is empty.
  	\item (Corollary \ref{existenceanduniqueness}) Otherwise, $M_{X, H}(\vv)$ is a reduced point, and the unique sheaf in $M_{X, H}(\vv)$ is locally free.
  \end{itemize}
\end{theorem}

In Example \ref{contraction} and \ref{P3}, we see both types of nodal $K3$ surfaces in Theorem \ref{main} exist. 

This paper is motivated by the study of exceptional bundles on $\mathbb{P}^{3}$. For any exceptional bundle $E$ on $\mathbb{P}^{3}$ with Chern character $\vv(E)$, \cite{Zub90} shows that $E|_{X}$ is an $H$-stable spherical vector bundles for a smooth quartic surface $i: X \hookrightarrow \mathbb{P}^{3}$, where $H$ is the hyperplane class of $\mathbb{P}^{3}$. In particular, $\vv$ satisfies $i^{*}(\vv(E))^{2}=-2$ under the Mukai pairing on $\mathrm{H}^{*}_{alg}(X)$. However, it is unknown that which spherical Mukai vectors on $X$ can be realized as $i^{*}(\vv(E))$ for some exceptional bundles $E$ on $\mathbb{P}^{3}$. If it can, then by a similar argument in \cite{Zub90}, $E|_{X_{0}}$ is a stable vector bundle for a very general singular quartic surface, which is a general nodal $K3$ surface under Definition \ref{nodalK3}. Hence the presence of stable spherical vector bundles on general singular quartic surfaces is a necessary condition for a spherical Mukai vector to lift. This motivation is discussed in Example \ref{P3}.

\subsection{Outline of the paper}
We sketch the strategy to prove Theorem \ref{main}. Let $X$ be a general nodal $K3$ surface and $\widetilde{X}$ its resolution. For a Mukai vector $\vv\in \HH^{*}_{alg}(X)$ with $\vv^{2}=-2$ and $\rk(\vv)>0$, we will find a stable spherical vector bundle $\widetilde{E}$ on $\widetilde{X}$ with Mukai vector $\vv$, and a condition for it to descend to $X$. In Section \ref{section3}, we show that this condition is equivalent to $\mathrm{Hom}_{\widetilde{X}}(\widetilde{E}, \widetilde{E}(L))=\mathbb{C}$, where $L$ is the exceptional divisor. In Section \ref{section4}, we compute this cohomology group by using Bridgeland stability. The idea for this computation is from a more general algorithm in \cite{Liu22}.

\subsection{Acknowledgments}
I want to thank Izzet Coskun, Benjamin Gould and Benjamin Tighe for many valuable advice and conversations. I want to thank Kota Yoshioka for many useful comments on an earlier version of this paper.

\section{Preliminaries and Notations}

\subsection{Stable spherical vector bundles on $K3$ surfaces}

In this subsection we collect necessary facts about stable spherical vector bundles on $K3$ surfaces. For details about $K3$ surfaces, we refer the readers to \cite{Huy16}. For details about stability of sheaves, we refer the readers to \cite{HL97}. First we note the following theorem, which follows from Theorem \ref{sphericalobject}.

\begin{theorem}\label{uniquesphericalvb}
  Let $Y$ be a  $K3$ surface and $\vv=(r, D, a)\in  \mathrm{H}^{*}_{alg}(Y)$ with $r>0$ and $\vv^{2}=-2$ under the Mukai pairing (such a Mukai vector $\vv$ is called \emph{spherical}). Then for a generic polarization $H$ (in the sense of \cite{O'G97}), the moduli space of $H$-Gieseker semistable sheaves $M_{Y, H}(\vv)$ 
  is a reduced point. Furthermore, the unique sheaf $E\in M_{Y, H}(\vv)$ is locally free. 
\end{theorem}

\subsection{Bridgeland stability condition}

The main tool in this paper is Bridgeland stability, which we briefly recall here. Some good references are \cite{Bri07, Bri08}. In this subsection, let $\mathcal{D}$ be the bounded derived category of a $K3$ surface $Y$.

\begin{definition}
  A \emph{(numerical) stability condition} on $\mathcal{D}$ is a pair $\sigma=(Z, \mathcal{A})$, where $\mathcal{A}$ is the heart of a bounded t-structure on $\mathcal{D}$, and
  $$Z: \KK_{\mathrm{num}}(Y):=\KK(Y)/\KK(Y)^{\perp} \longrightarrow \mathbb{C}$$
  is a homomorphism of abelian groups, called the \emph{central charge}, such that
  \begin{itemize}
  \item For any $0\neq E\in \mathcal{A}$, we have $Z(E)=\rho(E)\cdot \mathrm{exp}(i\pi \phi(E))$ for some $\rho(E)\in \mathbb{R}_{+}$ and the \emph{phase} $\phi(E)\in (0,1]$. 
  \item (Harder-Narasimhan filtration) An object $E\in \mathcal{A}$ is called $\sigma$-\emph{(semi)stable}, if for any subobject $0\neq F \subset E$ in $\mathcal{A}$, we have $\phi(F) \leqp \phi(E)$. For any object $E\in \mathcal{A}$, there exists a filtration of objects in $\mathcal{A}$:
    $$0=E_{0}\subset E_{1}\subset \cdots \subset E_{n}=E,$$
    such that $\mathrm{gr}_{i}(E)=E_{i}/E_{i-1}$ are $\sigma$-semistable and $\phi(\mathrm{gr}_{i})> \phi(\mathrm{gr}_{i+1})$ for all $i$. 
    \item (Support condition) For a fixed norm $|| \cdot ||$ on $\KK_{\mathrm{num}}(Y)\otimes \mathbb{R}$, there exists a constant $C>0$, such that 
    $$ Z(\vv(E)) \geq C ||\vv(E)|| $$
    for all $\sigma$-semistable objects $E\in \mathcal{A}$. 
  \end{itemize}
\end{definition}

\begin{theorem}[\cite{Bri07}]
The set of stability conditions on $Y$, denoted by $\mathrm{Stab}(Y)$, has a complex manifold structure.
\end{theorem}

In this paper we shall only use specific stability conditions constructed as follows. For any $\beta\in \mathrm{Pic}(Y)_{\mathbb{R}}$ and $\omega\in \mathrm{Amp}(Y)_{\mathbb{R}}$, define the following subcategories of $\mathcal{D}$:
\begin{itemize}
\item $\mathcal{F}_{(\beta, \omega)}=\{E\in \mathrm{Coh}(Y): \mbox{for all subsheaves }0\neq F\subset E, \dfrac{c_{1}(F)\cdot \omega}{r(F)} \leq \beta\cdot \omega \}$,
\item $\mathcal{T}_{(\beta, \omega)}=\{E\in \mathrm{Coh}(Y): E \mbox{ is torsion, or for all quotient }E\twoheadrightarrow F, \dfrac{c_{1}(F)\cdot \omega}{r(F)}>\beta\cdot \omega \}$.
\end{itemize}
Then $\mathcal{F}_{(\beta, \omega)}, \mathcal{T}_{(\beta, \omega)}$ is a torsion pair. Let
$$\mathcal{A}_{(\beta, \omega)}=\{E\in \mathcal{D}: \mathcal{H}^{-1}(E)\in \mathcal{F}_{(\beta, \omega)}, \mathcal{H}^{0}(E)\in \mathcal{T}_{(\beta, \omega)}, \mathcal{H}^{i}(E)=0 \mbox{ for }i\neq 0, -1\}$$
be the tilt of $\mathrm{Coh}(Y)$ with respect to this torsion pair. Define $Z_{(\beta, \omega)}: \KK_{\mathrm{num}}(Y) \longrightarrow \mathbb{C}$ as follows:
\begin{equation}\label{Z}
	Z_{(\beta, \omega)}(r, D, a)=-a-r \frac{\beta^{2}-\omega^{2}}{2}+ D\cdot \beta + i\omega\cdot (D-r\beta). 
\end{equation}
Then $\sigma_{(\beta, \omega)}=(Z_{(\beta, \omega)}, \mathcal{A}_{(\beta, \omega}))$ is a stability condition if $Z_{(\beta, \omega)}(\vv)\notin \mathbb{R}_{\leq 0}$ for all $\vv\in \KK_{\mathrm{num}}(Y)$ with $r(\vv)>0$ and $\vv^{2}=-2$ \cite{Bri08}. For every ample class $\omega\in \mathrm{Amp}(Y)_{\mathbb{R}}$, let
\begin{equation}\label{He}
	\mathbb{H}_{\omega}=\{\sigma_{(s,t)}=(Z_{(s\omega,t \omega)}, \mathcal{A}_{(s\omega,t\omega)} )\in \mathrm{Stab}(Y): s\in \mathbb{R}, t\in \mathbb{R}_{+}\}.
	\end{equation}	
Then $\mathbb{H}_{\omega}$ is an upper half plane with countably many line segments removed.

For a fixed Mukai vector $\vv$, there is a locally finite collection of real codimension 1 submanifolds of $\mathrm{Stab}(Y)$ called \emph{walls}. The connected components of the complement of walls are called \emph{chambers}. For stability conditions within the same chamber, the stable objects whose Mukai vectors are $\vv$ stay the same \cite{Bri08}. The intersections of walls for $\vv$ with $\mathbb{H}_{\omega}$ are nested semicircles \cite{Mac14}.

The following fact about stable spherical objects will be used later.

\begin{theorem}[\cite{BM14a, BM14b}]\label{sphericalobject}
  Let $\vv\in \mathrm{H}^{*}_{alg}(Y)$ satisfy $\rk(\vv)>0$ and $\vv^{2}=-2$. Then for a generic stability condition $\sigma$, $M_{\sigma}(\vv)$, the moduli space of $\sigma$-semistable objects in $\mathcal{A}$ with Mukai vector $\vv$, is a reduced point. 
\end{theorem}

We also note the following relation between Gieseker stability and Bridgeland stability.

\begin{theorem}[\cite{Bri08}, large volume limit]\label{largevolumelimit}
  Let $\omega\in \mathrm{Amp}(Y)_{\mathbb{R}}$, and let $E\in \mathrm{Coh}(Y)$ such that $\rk(E)>0$ and $c_{1}(E)\cdot \omega>0$. Then $E$ is $\omega$-Gieseker semistable if and only if $E$ is $\sigma_{(0, t\omega)}$-semistable for $t\gg 0$.
\end{theorem}

\section{Descending from the resolution}\label{section3}
Let $X$ be a general nodal $K3$ surface and $\pi: \widetilde{X} \longrightarrow X$ be its resolution. In this section, we study the condition for a vector bundle on $\widetilde{X}$ to descend to a vector bundle on $X$.

\begin{proposition}\label{descend}
  Let $\widetilde{E}$ be a vector bundle over $\widetilde{X}$. Then $\widetilde{E}=\pi^{*}(E)$ for some locally free sheaf $E$ over $X$ if and only if $\widetilde{E}|_{L}\cong \mathcal{O}^{\oplus r}$, where $r$ is the rank of $\widetilde{E}$.
\end{proposition}
We use the following lemma to prove Proposition \ref{descend}.

\begin{lemma}\label{surjection}
Assume $\widetilde{E}|_{L}=\mathcal{O}_{L}^{\oplus r}$. Then for any $m\in \mathbb{Z}_{>0}$, $\widetilde{E}$ fits into the following exact sequence
  $$0 \longrightarrow \widetilde{E}(-mL) \overset{i}{\longrightarrow} \widetilde{E} \overset{f_{m}}{\longrightarrow} Q_{m}  \longrightarrow 0, $$
  where $i$ is the composition of the natural inclusions
  $$\widetilde{E}(-mL)\subset \widetilde{E}(-(m-1)L)\subset \cdots \subset \widetilde{E}(-L)\subset \widetilde{E}.$$
  The sheaf $Q_{m}$ admits a map $g_{m}: Q_{m} \rightarrow \widetilde{E}|_{L}$, such that $g_{m} \circ f_{m}: \widetilde{E} \rightarrow \widetilde{E}|_{L} $ is the restriction map $f_{1}$, and the induced map on cohomology $\mathrm{H}^{0}(\widetilde{X}, Q_{m}) \rightarrow \mathrm{H}^{0}(\widetilde{X}, \widetilde{E}|_{L})$ is surjective. 
  
\end{lemma}

\begin{proof}
  We induct on $m$. When $m=1$, we have the following restriction sequence:
  $$0 \longrightarrow \widetilde{E}(-L) \longrightarrow \widetilde{E}  \overset{f_{1}}{\longrightarrow} \widetilde{E}|_{L}  \longrightarrow 0. $$ 
  The lemma follows by taking $Q_{1}=\widetilde{E}|_{L}$ and $g_{1}=\mathrm{id}$. 

  Suppose the lemma is true for $m-1$. Then we have the following commutative diagram
  \[\begin{tikzcd}
      & & 0\arrow[d] & 0\arrow[d] & \\
      0 \arrow[r] & \widetilde{E}(-mL) \arrow[r]\arrow[d,"="] & E(-(m-1)L) \arrow[r]\arrow[d] & \mathcal{O}_{L}(2(m-1))^{\oplus r} \arrow[r] \arrow[d] & 0\\
      0 \arrow[r] & \widetilde{E}(-mL) \arrow[r] & E \arrow[r, "f_{m}"]\arrow[d, "f_{m-1}"] & Q_{m} \arrow[r] \arrow[d, "\beta_{m}"] & 0\\
      & & Q_{m-1} \arrow[r, "="]\arrow[d] & Q_{m-1}\arrow[d] & \\
            & &0 & 0 & .
    \end{tikzcd}\]
Define $g_{m}:=g_{m-1}\circ \beta_{m}: Q_{m} \rightarrow \widetilde{E}|_{L} $. By the induction hypothesis, we have
$$g_{m} \circ f_{m} = g_{m-1} \circ \beta_{m} \circ f_{m} =g_{m-1} \circ f_{m-1}= f_{1}.$$
  Since $m\geq 1$, we have $ \mathrm{H}^{1}(\widetilde{X}, \mathcal{O}_{L}(2(m-1)))=0$. Hence $ (\beta_{m})_{*}: \mathrm{H}^{0}(\widetilde{X}, Q_{m})\twoheadrightarrow \mathrm{H}^{0}(\widetilde{X}, Q_{m-1})$ is surjective. By the induction hypothesis, the map $ \mathrm{H}^{0}(\widetilde{X}, Q_{m-1}) \twoheadrightarrow  \mathrm{H}^{0}(\widetilde{X}, \widetilde{E}|_{L})$ is surjective, the lemma is proved.
  
\end{proof}

Now we prove Proposition \ref{descend}. Let $H$ be the pullback of the ample generator of $\mathrm{Pic}(X)$. 

\begin{proof}[Proof of Proposition \ref{descend}]
  First note that if $\widetilde{E}=\pi^{*}(E)$, then $\widetilde{E}|_{L}=\pi^{*}(E|_{p})\cong \mathcal{O}_{L}^{\oplus r}$.

  Conversely, assume $\widetilde{E}|_{L}=\mathcal{O}_{L}^{\oplus r}$. For $n\gg 0$, $A=nH-L$ is ample on $\widetilde{X}$, we have
  $ \mathrm{H}^{1}(\widetilde{X}, \widetilde{E}(mA))=0$ for $m\gg 0$.
  Twisting the exact sequence in Lemma \ref{surjection} by $mnH$, we have
  $$0 \longrightarrow \widetilde{E}(mA)  \longrightarrow \widetilde{E}(mnH)  \longrightarrow Q_{m}(mnH)  \longrightarrow 0. $$
  Since $Q_{m}$ is supported on $L$ and $\mathcal{O}_{\widetilde{X}}(H)|_{L}=\mathcal{O}_{L}$, we have $Q_{m}(mnH)= Q_{m}$. Let $F=\widetilde{E}(mnH)$. Taking long exact sequence of cohomology, we get a surjective map $ \mathrm{H}^{0}(\widetilde{X}, F) \twoheadrightarrow  \mathrm{H}^{0}(\widetilde{X}, Q_{m})$. By Lemma \ref{surjection}, we have a surjective map $ \mathrm{H}^{0}(\widetilde{X}, Q_{m}) \twoheadrightarrow  \mathrm{H}^{0}(\widetilde{X}, F|_{L})$. Hence the following restriction map is surjective
  $$ \mathrm{H}^{0}(\widetilde{X}, F) \twoheadrightarrow  \mathrm{H}^{0}(\widetilde{X}, F|_{L})= \mathrm{H}^{0}(L, \mathcal{O}_{L}^{\oplus r})=\mathbb{C}^{\oplus r}.$$

  Let $s_{1}, \cdots, s_{r}\in  \mathrm{H}^{0}(\widetilde{X}, F)$ be lifts of a basis of $ \mathrm{H}^{0}(L, \mathcal{O}_{L}^{\oplus r})=\mathbb{C}^{\oplus r}$. Let $Z$ be the dependency locus of $s_{1}, \cdots, s_{r}$. Since $s_{1}, \cdots, s_{r}$ are independent on $L$, $Z\cap L=\emptyset$. Hence $\pi|_{Z}$ is an isomorphism onto its image. Let $U=X-\pi(Z)$ and $\widetilde{U}=\pi^{-1}(U)$. Then $F|_{\widetilde{U}}\cong \mathcal{O}_{\widetilde{U}}^{\oplus r}$. Since $p$ is an ordinary double point, in particular a rational singularity, by the projection formula we have
  $$\pi_{*}(\widetilde{E})=\pi_{*}(F|_{\widetilde{U}}(-mnH))=\pi_{*}(F|_{\widetilde{U}})\otimes \mathcal{O}_{U}(-mnH)=\pi_{*}(\mathcal{O}_{\widetilde{U}}^{\oplus r})(-mnH)=\mathcal{O}_{U}(-mnH)^{\oplus r}.$$

  Hence $E=\pi_{*}(\widetilde{E})$ is locally free. By adjunction of functors, we have $\mathrm{Hom}_{\widetilde{X}}(\pi^{*}\pi_{*}(\widetilde{E}),\widetilde{E})\cong \mathrm{Hom}_{X}(\pi_{*}(\widetilde{E}),\pi_{*}(\widetilde{E}))$. Hence there is a map $f: \pi^{*}E \longrightarrow \widetilde{E}$ that corresponds to the identity map in $\mathrm{Hom}_{X}(\pi_{*}(\widetilde{E}),\pi_{*}(\widetilde{E}))$. Since $f$ is an isomorphism everywhere, we have $\pi^{*}(E)\cong \widetilde{E}$. 

\end{proof}

Next we show an equivalent condition of Proposition \ref{descend} for a simple rigid vector bundle whose first Chern class is some multiple of $H$. 

\begin{proposition}\label{equivalent}
Let $\widetilde{E}$ be a simple rigid vector bundle on $\widetilde{X}$ with $c_{1}(\widetilde{E})=dH$, $d\in \mathbb{Z}$. Then $\widetilde{E}=\pi^{*}(E)$ for a locally free sheaf $E$ on $X$, if and only if $\mathrm{Hom}_{\widetilde{X}}(\widetilde{E}, \widetilde{E}(L))=\mathbb{C}$. 
\end{proposition}

\begin{proof}
In this proof, all cohomology computations are taken on $\widetilde{X}$ unless otherwise stated. Since $\widetilde{E}$ is rigid, we have $\mathrm{Ext}^{1}(\widetilde{E}, \widetilde{E})=0$. 
  Applying $\mathrm{Hom}(\widetilde{E},-)$ to the short exact sequence
  $$0 \longrightarrow \widetilde{E} \longrightarrow \widetilde{E}(L) \longrightarrow \widetilde{E}(L)|_{L} \longrightarrow 0 ,$$
  we have the long exact sequence
  $$0 \longrightarrow \mathrm{Hom}(\widetilde{E}, \widetilde{E}) \longrightarrow \mathrm{Hom}(\widetilde{E}, \widetilde{E}(L))  \longrightarrow \mathrm{Hom}(\widetilde{E}, \widetilde{E}(L)|_{L}) \longrightarrow \mathrm{Ext}^{1}(\widetilde{E}, \widetilde{E})=0. $$
  Since $\widetilde{E}$ is simple, $\mathrm{Hom}(\widetilde{E},\widetilde{E})=\mathbb{C}$. Hence $\mathrm{Hom}(\widetilde{E},\widetilde{E}(L))=\mathbb{C}$ if and only if
  \begin{equation}\label{1}
  \mathrm{Hom}(\widetilde{E},\widetilde{E}(L)|_{L})\cong \mathrm{Hom}_{L}(\widetilde{E}|_{L}, \widetilde{E}|_{L}(-2))=0. 
\end{equation}

  If $\widetilde{E}$ descends, then by Proposition \ref{descend}, $\widetilde{E}(L)|_{L}\cong \mathcal{O}_{L}(-2)^{\oplus r}$. Hence
  $$\mathrm{Hom}(\widetilde{E},\widetilde{E}(L)|_{L})=\mathrm{Hom}_{L}(\widetilde{E}|_{L}, \widetilde{E}|_{L}(-2))=\mathrm{Hom}_{L}(\mathcal{O}_{L}^{\oplus r}, \mathcal{O}_{L}(-2)^{\oplus r})=0. $$

  Conversely, assume $\widetilde{E}$ cannot descend to a locally free sheaf. Since $\widetilde{E}$ is locally free, we may write $\widetilde{E}|_{L}=\bigoplus_{k=1}^{r}\mathcal{O}_{L}(a_{k}), a_{1}, \cdots, a_{r}\in \mathbb{Z}$. By Proposition \ref{descend}, $a_{k}$ cannot all be zero. Since $c_{1}(E)=dH\in \mathrm{Pic}(X)$ and $H\cdot L=0$, we have $\sum_{k=1}^{r}a_{k}=0$. Hence there exist some $a_{i}>0$ and some $a_{j}<0$, in particular $a_{j}\leq a_{i}-2$. Then we have
  $$\mathrm{hom}_{L}(\widetilde{E}|_{L}, \widetilde{E}|_{L}(-2))=\mathrm{hom}_{L}(\bigoplus_{k=1}^{r} \mathcal{O}_{L}(a_{k}), \bigoplus_{k=1}^{r} \mathcal{O}_{L}(a_{k}-2))\geq \mathrm{hom}_{L}(\mathcal{O}_{L}(a_{j}), \mathcal{O}_{L}(a_{i}-2))>0. $$
  Hence (\ref{1}) is not satisfied, we have $\mathrm{Hom}(\widetilde{E}, \widetilde{E}(L))\neq \mathbb{C}$.
\end{proof}

\section{Main Results}\label{section4}

In this section we prove the following main result of this paper. 

\begin{theorem}[Existence]\label{existence}
  Let $X$ be a general nodal $K3$ surface, $H$ be the ample generator of $\mathrm{Pic}(X)$, and $\vv=(r, dH, a)\in  \mathrm{H}^{*}_{alg}(X)$ with $\vv^{2}=-2$ and $r>0$. If either
\begin{enumerate}
\item $\mathrm{Cl}(X)= \mathrm{Pic}(X)$, or
\item $r\neq 2$,
\end{enumerate}
  then for $0 < \varepsilon \ll 1$, the unique sheaf $\widetilde{E}\in M_{\widetilde{X}, H-\varepsilon L}(\vv)$ descends to a locally free sheaf $E\in M_{X,H}(r, dH, a)$, where $L$ is the exceptional divisor. 
\end{theorem}

We scketch the idea of the proof. Under the conditions in Theorem \ref{existence}, we will find a wall (denoted by $W_{-1}$) in $\mathrm{Stab}(\widetilde{X})$ and adjacent chambers $\mathcal{C}_{+}, \mathcal{C}_{-}$ of $W_{-1}$ (see Figure \ref{setup}), such that for any $F_{+}\in M_{\widetilde{X}, \mathcal{C}_{+}}(\vv)$, the $\mathcal{C}_{-}$-Harder-Narasimhan filtration of $F_{+}$ is
$$0 \longrightarrow \widetilde{E} \longrightarrow F_{+} \longrightarrow \mathcal{O}_{L}(-2)^{\oplus r} \longrightarrow 0 .$$
Theorem \ref{existence} will follow from Proposition \ref{equivalent} if we prove $F_{+} \cong \widetilde{E}(L)$. 
To do this, we show that there are no walls between $\mathcal{C}_{+}$ and the $(H-\varepsilon L)$-Gieseker chamber of $v$ for $0<\varepsilon \ll 1$, by using the numerical conditions for walls developed in \cite{BM14a}. 

Recall that for a general nodal $K3$ surface $X$ (Definition \ref{nodalK3}) and its resolution $\widetilde{X}$, we have $\mathrm{Pic}(X)\oplus \mathbb{Z}L=\mathbb{Z}H\oplus \mathbb{Z}L \hookrightarrow \mathrm{Pic}(\widetilde{X})$ as a full rank sublattice. The pullback map $\pi^{*}$ induces an embedding $\pi^{*}: \mathrm{Pic}(X) \longrightarrow \mathrm{Pic}(\widetilde{X})$. We will sometimes write $H$ for $\pi^{*}(H)$ on $\widetilde{X}$ if the context is clear. For a fixed Mukai vector $v\in  \mathrm{H}^{*}_{alg}(\widetilde{X})$, we have the chamber decomposition of $\mathrm{Amp}(\widetilde{X})_{\mathbb{R}}$ that distinguishes different Gieseker stability conditions \cite{HL97}. 

Since $X$ is a general nodal $K3$ surface, we have $\mathrm{Pic}(\widetilde{X})_{\mathbb{R}}=\mathbb{R}H\oplus \mathbb{R}L$. Hence the walls are rays from the origin. Since $H\in \mathrm{Pic}(\widetilde{X})$ generates an extremal ray of $\mathrm{Nef}(\widetilde{X})_{\mathbb{R}}$, there is exactly one chamber of $\mathrm{Amp}(\widetilde{X})_{\mathbb{R}}$ whose boundary contains $H$. For $0< \varepsilon \ll 1$, $(H-\varepsilon L)$ belongs to this chamber.

Before proving the theorem we set some notations. 

\begin{notation}\label{notation1}
For $0<\varepsilon \ll 1$, let $H_{\varepsilon}=H-\varepsilon L$ and $\mathbb{H}_{\varepsilon}=\mathbb{H}_{H_{\varepsilon}}$ (see (\ref{He})). For $0<\varepsilon'\ll \varepsilon \ll 1$, set $s(\varepsilon')=\frac{dH\cdot H_{\varepsilon}}{r H_{\varepsilon}^{2}}-\varepsilon'$. Let
$$\mathbf{b}=\mathbf{b}_{\varepsilon'}=\{\sigma_{(s, t)}: s=s(\varepsilon'), t>0 \}\subset \mathbb{H}_{\varepsilon}, $$ and $\mathcal{A}=\mathcal{A}_{s(\varepsilon')H_{\varepsilon}}$. Let 
$$\uu=(r, dH, a),~ \vv=(r, dH+rL, a-r),~ \ttt_{m}=(0, L, m),~ \uu, \vv,\ttt\in  \mathrm{H}^{*}_{alg}(\widetilde{X}). $$
 Let $W_{m}=W(\uu, \ttt_{m})$ be the numerical wall defined by $\uu$ and $\ttt_{m}$, and $\sigma_{+}\in \mathbf{b}$ (resp. $ \sigma_{-}\in \mathbf{b}$) be a generic stability condition that is above (resp. below) but near $W_{-1}$, and $\sigma_{0}= W_{-1}\cap \mathbf{b}$. (See Figure \ref{setup}.)
\end{notation}

\begin{definition}\label{blackhole}
  For the spherical Mukai vector $\uu=(r, dH, a)$, we define the point
$$\sigma_{\uu}=\left(\frac{dH^{2}}{rH_{\varepsilon}^{2}}, \sqrt{\frac{2a}{rH_{\varepsilon}^{2}}-\left(\frac{dH^{2}}{rH_{\varepsilon}^{2}} \right)^{2}} \right)$$
inside the $(s,t)$-upper half plane which contains $\mathbb{H}_{\varepsilon}$. (See Figure \ref{setup}.)
\end{definition}
\begin{figure}[h]
	\centering{
		\resizebox{90mm}{!}{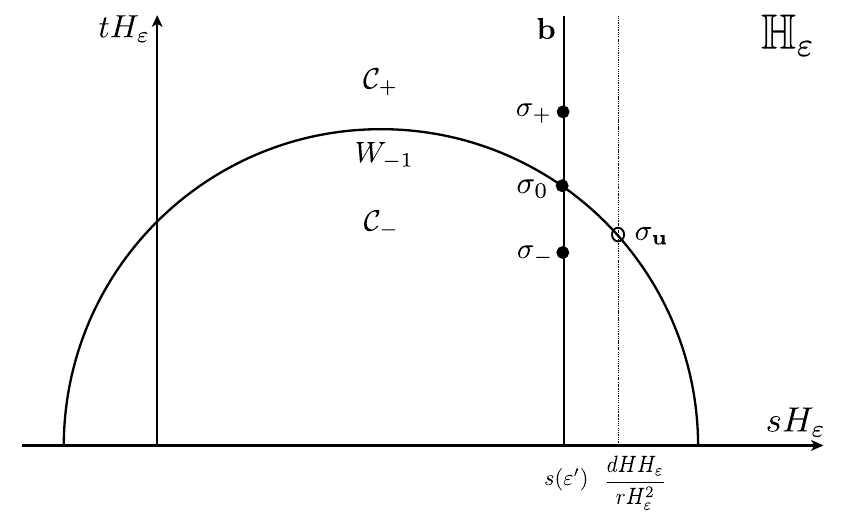}
		\caption{$\sigma_{+}\in \mathcal{C}_{+}, \sigma_{-}\in \mathcal{C}_{-}$.}
		\label{setup}
	}
\end{figure}

Note that $\sigma_{\uu}\notin \mathrm{Stab}(\widetilde{X})$: extending the central charge (see (\ref{Z})) to the $(s,t)$-upper half plane, we have $Z_{\sigma_{\uu}}(\uu)=0$. Hence the closure of every numerical wall of $\uu$ contains $\sigma_{\uu}$. 

The equation of $W_{m}$ is given by
$$\varepsilon r(s^{2}+t^{2})H_{\varepsilon}^{2}-(rmH_{\varepsilon}^{2})s=-mdH^{2}+2\varepsilon a.$$
For all $m\in \mathbb{Z}$, $W_{m}$ contains $\sigma_{\uu}$. The slope of $W_{m}$ at $\sigma_{\uu}$ is
$$
\der{W_{m}}{s}=\frac{rmH_{\varepsilon}^{2}-2\varepsilon dH^{2}}{2\varepsilon rt H_{\varepsilon}^{2}}. $$

The following two lemmas find the two $\sigma_{0}$-stable objects in $\mathcal{A}$. Recall that $\widetilde{E}$ is defined in Theorem \ref{existence}.

\begin{lemma}\label{stabilityS}
  The sheaf $\widetilde{E}\in \mathcal{A}$ is $\sigma_{0}$-stable. 
\end{lemma}

\begin{proof}
  First note that $\widetilde{E}\in \mathcal{A}$, since 
  $$ \mathrm{Im}[Z_{(s_{\varepsilon'}H_{\epsilon}, tH_{\varepsilon})}(\widetilde{E})]= t \varepsilon' r H_{\epsilon}^{2} >0, ~ \forall t\in \mathbb{R}_{+}.$$
	
  Since $H'$ and $H_{\varepsilon}$ belong to the same chamber, $\widetilde{E}$ is $H_{\varepsilon}$-stable. Since $\uu^{2}=-2$, $\mathrm{gcd}(r,d)=1$, $\widetilde{E}$ is $\mu_{H_{\varepsilon}}$-stable. By Theorem \ref{largevolumelimit}, $\widetilde{E}$ is $\sigma_{s_{u}, t}$-stable for $t\gg 0$. Since $\uu,\ttt_{-1}$ are the minimal solutions for the Pell equation
  $$\ww^{2}=-2, \ww\in \langle \uu, \ttt_{-1} \rangle \subset \mathrm{H}^{*}_{alg}(\widetilde{X}), $$
  $W_{-1}$ is not an actual wall of $\uu$. Hence it suffices to prove that there is no actual wall of $\uu$ on $\mathbf{b}$ that is above $W_{-1}$.

  Suppose $W$ is such an actual wall. By \cite{BM14a}, there exists a spherical Mukai vector $u'\neq 0$ such that $\uu'$ and $\uu-\uu'$ are $\mathcal{A}$-effective and $\uu\uu'<0$. Since $W$ is a numerical wall of $\uu$, $W$ contains $\sigma_{0}$. Since $\mathbf{b}$ is on the left of $\sigma_{0}$ and $W$ is above $W_{-1}$ on $\mathbf{b}$, we have
  \begin{equation}\label{derivative}
  \der{W}{t}<\der{W_{-1}}{t}=\frac{-rH_{\varepsilon}^{2}-2\varepsilon dH^{2}}{2\varepsilon rt H_{\varepsilon}^{2}}. 
  \end{equation}
  Since $\mathrm{lim}_{\varepsilon \rightarrow 0}\frac{dW_{-1}}{dt}=-\infty $, the wall $W$ becomes vertical as $\varepsilon \rightarrow 0$. Hence we must have $\mu_{H}(\uu')=\mu_{H}(\uu)$. Then $\uu'=(kr, kdH+eL, m)$ for some $k, e\in \mathbb{Q}$ and $m\in \mathbb{Z}$. Since $p$ is an ordinary double point, its class group is $\mathbb{Z}/2 \mathbb{Z}$. Hence the pushforward of $2(kdH+eL)$ is Cartier on $X$, we have $k, e\in \frac{1}{2} \mathbb{Z}$. Write $k_{1}=2k, e_{1}=2e$ for $k_{1}, e_{1} \in \mathbb{Z}$. Summarizing the conditions for $k_{1}, e_{1}$, we have
  \begin{equation}\label{numerical}
  \begin{array}{c}
  	\uu'=\left(\dfrac{k_{1}}{2}r, \dfrac{k_{1}dH+e_{1}L}{2}, m\right), (\uu')^{2}=-2, \\
  	\mathrm{Im}[Z_{(s_{\varepsilon'}H_{\epsilon}, t H_{\epsilon})}(\uu')]>0, ~ \mathrm{Im}[Z_{(s_{\varepsilon'}H_{\epsilon}, t H_{\epsilon})}(\uu-\uu')]>0, ~\forall t\in \mathbb{R}_{+},
  \end{array}
\end{equation}
  where $\mathrm{Im}(-)$ denotes the imaginary part of a complex number, and $Z(-)$ is defined in (\ref{Z}).
  We will show that under (\ref{derivative}) and (\ref{numerical}), $k=1, e=0$. Then $\uu'=\uu$ and $W=W_{-1}$.
  \\
  \textbf{Step 1: When $e=0$, (\ref{numerical}) implies $k=1$.}
  First observe that $k_{1}$ and $e_{1}$ must have the same parity. Since $\frac{k_{1}dH+e_{1}L}{2}\in \mathrm{Pic}(\widetilde{X})$, we have
  $$\left( \frac{k_{1}dH+e_{1}L}{2} \right)^{2}=\frac{k_{1}^{2}d^{2}H^{2}-2e_{1}^{2}}{4}\in \mathbb{Z}. $$
  If $e_{1}$ is odd, then $k_{1}$ is odd. Assume $e_{1}$ is even and $k_{1}$ is odd. Since $\frac{k_{1}r}{2}\in \mathbb{Z} $, $r$ is even. Since $\mathrm{gcd}(r,d)=1$, $d$ is odd. Since $\frac{k_{1}^{2}d^{2}H^{2}}{4}\in \mathbb{Z} $, we have $4|H^{2}$. On the other hand, since $\uu=(r, dH, a)$ is spherical, we have
  $$-2=d^{2}H^{2}-2ra \equiv 0 (\mathrm{mod}~ 4),$$
  a contradiction. 
  
  Since $k_{1}, e_{1}$ have the same parity, $k\in \mathbb{Z}$. Since $(\uu')^{2}=-2$, we have
  $$-2=k^{2}d^{2}H^{2}-2krm, $$
  hence $k|2$. The right hand side is not divisible by 4, hence $k$ is odd. 
  We have $k=\pm 1$ and $\uu'=\pm \uu$. Since $\uu'$ is $\mathcal{A}$-effective, we have $k=1$. Hence in the following we assume $e\neq 0$.
  \\
  \textbf{Step 2: When $e\neq 0$, (\ref{numerical}) implies $k\geq 0$.}
  Since $\uu-\uu'$ is $\mathcal{A}$-effective, we have
  \begin{equation}\label{2}
  	0\leq H_{\varepsilon}((1-k)dH-eL-(1-k)rs(\varepsilon')H_{\varepsilon})=(1-k)r\varepsilon'H_{\varepsilon}^{2}-2e\varepsilon. 
  	\end{equation}
  Note that $k$ and $e$ cannot both be negative: since $\uu'$ is $\mathcal{A}$-effective, we have
  $$0\leq H_{\varepsilon}((kdH+eL)-krs(\varepsilon')H_{\varepsilon})=2e\varepsilon + \varepsilon' krH_{\varepsilon}^{2}.$$
  If $k<0$, then $\frac{1-k}{e}\geq \frac{2\varepsilon}{r\varepsilon' H_{\varepsilon}^{2}}\gg 0 $. Hence $\frac{k^{2}}{e^{2}} \gg 0$. On the other hand, since $\uu\uu'<0$, we have
  $$0>\uu\uu'=kd^{2}H^{2}-kra-rm=-k+\frac{e^{2}-1}{k}. $$
  If $k<0$, then $k^{2}<e^{2}$, a contradiction. Hence $k\geq 0$.
  \\
  \textbf{Step 3: When $e\neq 0, k\geq 0$, (\ref{derivative}) and (\ref{numerical}) imply $e>0$.}
  Since $\uu'$ is $\mathcal{A}$-effective, we have
  $$0\leq H_{\varepsilon}((kdH+eL)-krs(\varepsilon')H_{\varepsilon})=2e\varepsilon + \varepsilon' krH_{\varepsilon}^{2}.$$
  If $e<0$, then $\frac{k}{-e}\geq \frac{2\varepsilon}{\varepsilon'rH_{\varepsilon}^{2}}\gg 0 $. On the other hand, $W=W(\uu, k\uu-\uu')$, and
  $$k\uu-\uu'=(0,-eL, ka-m )=\left(0, -eL, \frac{k}{r}+\frac{e^{2}-1}{k}\right)=-e\left(0, L, \frac{k}{-er}+\frac{e^{2}-1}{-ekr}\right).$$
  Let $l=\frac{k}{-er}+\frac{e^{2}-1}{-ekr}$. Then
  $$\der{W}{s}=\frac{rlH_{\varepsilon}^{2}-2\varepsilon dH^{2}}{2\varepsilon rtH_{\varepsilon}^{2}}\gg 0. $$
  Since $W$ is above $W_{-1}$ on $\mathbf{b}$, we have $\der{W}{s}<\der{W_{-1}}{s}<0 $, a contradiction. Hence $e> 0$. 
  \\
  \textbf{Step 4: When $k\geq 0, e>0$, (\ref{numerical}) is not satisfied.}
  By (\ref{2}), we have
  $$0\leq (1-k)r\varepsilon'H_{\varepsilon}^{2}-2e\varepsilon \leq r\varepsilon'H_{\varepsilon}^{2}-2e\varepsilon.$$
  However, since $e> 0$, $0<\varepsilon'\ll \varepsilon$, and $r$ is independent of the choice of $\varepsilon, \varepsilon'$, we have $r\varepsilon'H_{\varepsilon}^{2}-2e\varepsilon <0$. This is a contradiction. 
\end{proof}

\begin{lemma}\label{stabilityT}
If either $\mathrm{Pic}(X)= \mathrm{Cl}(X)$ or $r=\rk(\widetilde{E})\neq 2$, then $\mathcal{O}_{L}(-2)\in \mathcal{A}$ is $\sigma_{0}$-stable.
\end{lemma}

\begin{proof}
  Recall that $\mathbf{b}=\mathbf{b}_{\varepsilon'}=\{\sigma_{(s, t)}: s=s(\varepsilon'), t>0 \}$ and $\sigma_{0}=\sigma_{0}(\varepsilon, \varepsilon')= W_{-1}\cap \mathbf{b}$ (Notation \ref{notation1}).
  
   Let $G$ be any $\sigma_{0}$-Jordan-H\"{o}lder factor of $\mathcal{O}_{L}(-2)$, and let $\mathbf{g}=(r', d'H+e'L, a')$ be its Mukai vector, where $r', a'\in \mathbb{Z}$ and $d', e'\in \frac{1}{2}\mathbb{Z} $. Then $\mathbf{g}, \ttt_{-1}-\mathbf{g}$ are both $\mathcal{A}$-effective. By Mukai's Lemma (see e.g. \cite{Bri08}), $G$ is simple and rigid. Hence $\mathbf{g}^{2}=-2$. Summarizing the conditions on $r', d', e', a'$, we have
   \begin{equation}\label{numerical2}
   	\begin{array}{c}
   		\mathbf{g}=(r', d'H+e'L, a'), ~\mathbf{g}^{2}=-2,\\
   		\mathrm{Im}[Z_{(s_{\varepsilon'}H_{\epsilon}, t H_{\epsilon})}(\mathbf{g})]>0, ~ \mathrm{Im}[Z_{(s_{\varepsilon'}H_{\epsilon}, t H_{\epsilon})}(\ttt_{-1}-\mathbf{g})]>0, ~\forall t\in \mathbb{R}_{+}.
   	\end{array}
   \end{equation}
   We claim that under (\ref{numerical2}), either $r'=0$ or $\mathbf{g}=\mathbf{u}$. The claim implies the lemma in the following way. If $\mathcal{O}_{L}(-2)$ is not $\sigma_{0}$-stable, then it has at least two $\sigma_{0}$-Jordan-H\"{o}lder factors. Since $\mathbf{u}$ has positive rank and $\mathbf{t}_{-1}$ has rank 0, $\mathbf{t}_{-1}=m\mathbf{g}$ for some $m\geq 2$, where $\mathbf{g}$ has rank 0. However, since $\mathbf{g}$ is spherical, we have 
   $$ -2=\mathbf{t}_{-1}^{2}=(m\mathbf{g})^{2}=-2m^{2}. $$
   Hence $m=1$, a contradiction. Now we prove the claim in several steps.
   \\
   \textbf{Step 1: Show $d'r-r'd=0$.} Since $\sigma_{0}=W_{-1}\cap \mathbf{b}_{\varepsilon'}$, $\sigma_{0}$ is generic on the wall $W_{-1}$. Hence $\mathbf{u}, \mathbf{g}, \mathbf{t}_{-1}$ are linearly dependent. We may write $\mathbf{g}=e'\mathbf{t}_{-1} + k \mathbf{u}$ for some $k\in \mathbb{Q}$. Hence $d'=kd, r'=kr$, and we have $d'r-r'd=0$.
   \\
   \textbf{Step 2: Either $e'=\frac{1}{2}$, $e'=0$, or $r'=0$.} Since $d'r -r'd=0$, $\mathrm{Im}[Z_{(s_{\varepsilon'}H_{\epsilon}, t H_{\epsilon})}(\mathbf{g})]>0$ in (\ref{numerical2}) implies that 
   \begin{equation*}\label{3}
   	(d'H-\dfrac{r'}{r}dH^{2})+2\varepsilon e' +r'\varepsilon' H_{\varepsilon}^{2} = 2\varepsilon e' +r'\varepsilon' H_{\varepsilon}^{2} \geq 0. 
   \end{equation*}
   Since $0 \ll \varepsilon' \ll \varepsilon <1$, we have $e'\geq 0$. 
   Similarly, $\mathrm{Im}[Z_{(s_{\varepsilon'}H_{\epsilon}, t H_{\epsilon})}(\ttt_{-1}-\mathbf{g})]>0$ in (\ref{numerical2}) implies that
   \begin{equation*}\label{4}
   	(-d'H^{2}+\dfrac{r'}{r}dH^{2})+2\varepsilon (1-e')-r'\varepsilon'H_{\varepsilon}^{2} =
   	2\varepsilon (1-e')-r'\varepsilon'H_{\varepsilon}^{2} \geq 0.
   \end{equation*}
   Since $0< \varepsilon' \ll \varepsilon$, we have $e'\leq 1$. Since $e'\in \frac{1}{2}\mathbb{Z}$, the possibilities are $e'=0, \frac{1}{2}, 1$. When $e'=1$, we must have $r'=0$.
   \\
   \textbf{Step 3: Show $e'\neq \frac{1}{2}$.} When $\mathrm{Pic}(X)= \mathrm{Cl}(X)$, this is automatically true. Hence for the rest of Step 3, we assume $\mathrm{Pic}(X)\neq \mathrm{Cl}(X)$ and $r\neq 2$. Suppose $e'=\frac{1}{2}$. First note that $k\in \frac{1}{2}\mathbb{Z}$. Since $d', e'\in \frac{1}{2}\mathbb{Z}$ and 
   $$ 2\mathbf{g}=(2kr, 2kdH+L, -2+ka), $$
   we have $2kr\in \mathbb{Z}, 2kd\in \mathbb{Z}$. Since $\gcd(r,d)=1$, we have $k\in \frac{1}{2}\mathbb{Z}$.
   
   Then we may write $\mathbf{g}=\frac{1}{2}\mathbf{t}_{-1}+ \frac{k_{1}}{2} \mathbf{u}$ for some $k_{1}\in \mathbb{Z}$. 
   Since $\mathbf{t}_{-1}\cdot \mathbf{u}=r$, we have 
   $$ -2=\mathbf{g}^{2}=\left(\frac{1}{2}\mathbf{t}_{-1}+ \frac{k_{1}}{2} \mathbf{u}\right)^{2}=-\frac{1}{2} + \frac{k_{1}}{2}r - \frac{k_{1}^{2}}{2}. $$
   Hence $k_{1}(k_{1}-r)=3$, in particular $k_{1}|3$. The possibilities of $k_{1}$ are $k_{1}=-3, -1, 1, 3$. Since $k_{1}$ is odd and $r'=\frac{k_{1}}{2}r \in \mathbb{Z}$, $r$ must be even. 
   \\
   \textit{When $k_{1}<0$:} Since $r\geq 2$, we have 
      $3= k_{1}(k_{1}-r)\geq k_{1}^{2}+2 .$
      Hence $k_{1}=-1, r=2$, which is excluded by the assumption.
      \\
      \textit{When $k_{1}=1$:} Since $r\geq 2$, $3=k_{1}(k_{1}-r)<0$, a contradiction. 
      \\
      \textit{When $k_{1}=3$:} In this case we have $r=2$, which is excluded by the assumption.
  \\
  \textbf{Step 4: Either $\mathbf{g}=\mathbf{u}$ or $r'=0$.} By Step 2, when $e'=1$, we have $r'=0$. By Step 3, $e'\neq \frac{1}{2}$. Hence we need to show when $e'=0$, $\mathbf{g}=\mathbf{u}$. In this case, $\mathbf{g}=\frac{k_{1}}{2}\mathbf{u}$. Hence 
  $$-2=\mathbf{g}^{2}=\left(\frac{k_{1}}{2}\mathbf{u}\right)^{2}=-\frac{k_{1}^{2}}{2},$$
  we have $\mathbf{g}=\frac{k_{1}}{2}=\pm 1$, $\mathbf{g}=\pm \mathbf{u}$. 
  Note that $\mathcal{E}\in \mathcal{A}$, we see $-\mathbf{u}$ is not $\mathcal{A}$-effective. Since $\mathbf{g}$ is $\mathcal{A}$-effective, we have $\mathbf{g}=\mathbf{u}$.   
\end{proof}

Now we prove Theorem \ref{existence}. 

\begin{proof}[Proof of Theorem \ref{existence}]

  If $r=1$, then $E=\mathcal{O}_{\widetilde{X}}(dH)=\pi^{*}(\mathcal{O}_{X}(dH))$. Hence in the following we assume $r\geq 2$. 
  
  By Lemma \ref{stabilityS} and Lemma \ref{stabilityT}, $\widetilde{E}, \mathcal{O}_{L}(-2)\in \mathcal{A}$ are the two stable spherical objects at the numerical wall $W_{-1}=W(\widetilde{E}, \mathcal{O}_{L}(-2))$ whose Mukai vectors are in the rank 2 lattice spanned by $\uu$ and $\ttt_{-1}$.
  Recall that $\vv=(r, dH+rL, a-r)$. 
  By \cite{BM14a}, $W_{-1}$ is an actual wall for $\vv$. Let $\sigma_{+}$(resp. $\sigma_{-}$) be a stability condition right above(resp. below) $W_{-1}$. Let $F_{+}\in M_{\sigma_{+}}(\vv)$. By Corollary 4.13 \cite{Liu22}, the $\sigma_{-}$-Harder-Narasimhan filtration of $F_{+}$ is
  $$0 \longrightarrow \widetilde{E} \longrightarrow F_{+} \longrightarrow \mathcal{O}_{L}(-2)^{\oplus r} \longrightarrow 0 .$$
  Hence $\mathrm{Hom}(\widetilde{E}, F_{+})=\mathbb{C}$. By Proposition \ref{equivalent}, it suffices to prove $\widetilde{E}(L)=F_{+}$. Since $H'$ and $H_{\varepsilon}$ are in the same chamber, $\widetilde{E}(L)$ is $H_{\varepsilon}$-stable. By Theorem \ref{largevolumelimit}, it suffices to prove that there is no wall on $\mathbf{b}$ above $W_{-1}$.

  Suppose on $\mathbf{b}$ there is a wall $W$ above $W_{-1}$, by \cite{BM14a}, there exists an $\mathcal{A}$-effective spherical Mukai vector $\vv'\neq \vv$ that defines $W$, and $\vv\vv'<0$. When $\varepsilon$ approaches zero, the walls $W$ and $W'$ are both becoming vertical on $\mathbb{H}_{\varepsilon}$. Hence we have $\mu_{H}(\vv')=\mu_{H}(\vv)$. Then we may write $\vv'=(kr, kdH+eL, m)$, where $ m\in \mathbb{Z}$ and $k=\frac{k_{1}}{2} \in \frac{1}{2}\mathbb{Z} , e=\frac{e_{1}}{2} \in \frac{1}{2}\mathbb{Z}$. As in the proof of Lemma \ref{stabilityS}, $k_{1}, e_{1}$ have the same parity. Also note that $\vv'$ and $\vv-\vv'$ are both $\mathcal{A}$-effective. Summarizing, we have 
  \begin{equation}
  	\begin{array}{c}
  		\mathbf{v}'=\left(\dfrac{k_{1}}{2}r,~\dfrac{k_{1}}{2}dH+\dfrac{e_{1}}{2}L, m\right), ~(\vv ')^{2}=-2, \vv \vv'<0, \\
  		\mathrm{Im}[Z_{(s_{\varepsilon'}H_{\epsilon}, t H_{\epsilon})}(\mathbf{v})]>0, ~ \mathrm{Im}[Z_{(s_{\varepsilon'}H_{\epsilon}, t H_{\epsilon})}(\vv-\vv ')]>0, ~\forall t\in \mathbb{R}_{+}.
  	\end{array}
  \end{equation}
We claim that the wall $W=W(\vv, \vv')$ is not above $W_{-1}$. 

First we assume $e\neq 0$. Suppose $e=0$, then $k\in \mathbb{Z}$. Since $(\uu')^{2}=-2$ and $m\in \mathbb{Z}$, we have $k=\pm 1$ and $\uu'=\pm \uu$, $W=W_{-1}$. Now we prove the claim in the following steps. 
\\
  \textbf{Step 1: Show $k\neq 0$.} Suppose $k=0$, then $\vv'=(0, eL, m)$. Since $(\vv')^{2}=-2$ and $\vv'$ is $\mathcal{A}$-effective, $e=1$.
  Since
  $$0<\vv\vv'=(r,dH+rL, a-r)\cdot (0, L, m)=(-2-m)r,$$
  we have $m>-2$. If $m=-1$, then $W=W_{-1}$, hence $m\geq 0$. We first show that $$\phi_{(s(\varepsilon), t)}(\vv') > \phi_{(s(\varepsilon), t)}(\vv)> \phi_{(s(\varepsilon), t)}(\uu) $$
  when $t\gg 0$. 
  To prove the first inequality, note that $\mathrm{lim}_{t \rightarrow \infty} \phi_{(s(\varepsilon), t)}(\vv)=0$. However, for $m\geq 0$ we have
  $$\mathrm{lim}_{t \rightarrow \infty} \phi_{(s(\varepsilon), t)}(\vv')
  =\mathrm{lim}_{t \rightarrow \infty} \phi_{(s(\varepsilon), t)}(0,L,m)
  =
  \bigg\{\begin{array}{l}
  	1/2, ~m=0,\\
  	1, ~m>0.
  \end{array}
  $$
  Hence $\phi_{(s(\varepsilon), t)}(\vv') > \phi_{(s(\varepsilon), t)}(\vv)$.
  The second inequality holds, because $\widetilde{E}, \widetilde{E}(L)$ are both $\mu_{H_{\varepsilon}}$-stable and $\widetilde{E}\subset \widetilde{E}(L)$. By Theorem \ref{largevolumelimit}, we have $\phi_{(s(\varepsilon), t)}(\widetilde{E}(L))> \phi_{(s(\varepsilon), t)}(\widetilde{E})$ when $t\gg 0$.

  Let $(s(\varepsilon), t_{1}), (s(\varepsilon), t_{2}), (s(\varepsilon), t_{3})$ be the intersection points of $\mathbf{b}$ with $W(\vv, \vv'), W(\vv,\uu)=W_{-1}, W_{\uu, \vv'}=W_{m}$ respectively. We may assume $\uu, \vv, \vv'$ are independent Mukai vectors, otherwise $W(\vv, \vv')=W(\vv, \uu)=W_{-1}$. Hence $W(\vv', \vv)$ does not pass through $\sigma_{\uu}$ (Definition \ref{blackhole}). By assumption $W(\vv', \vv)$ is above $W_{-1}$, we have $t_{1}>t_{2}$ and $t_{1}>t_{3}$. Since $m\geq 0$, we have $t_{1}> t_{2}> t_{3}$. (See Figure \ref{walls}.)
  \begin{figure}[h]
  	\centering{
  		\resizebox{90mm}{!}{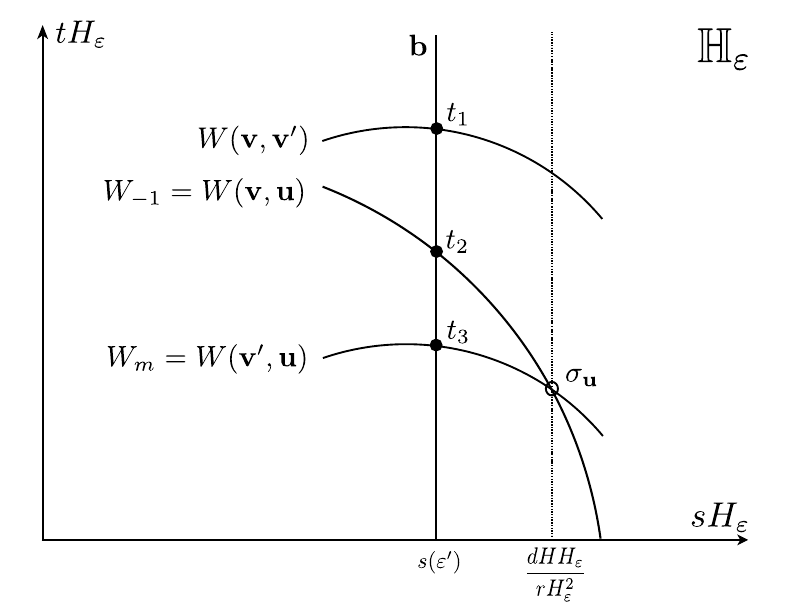}
  		\caption{$W(\vv, \vv')$, $W(\vv, \uu)$, and $W(\vv', \uu)$.}
  		\label{walls}
  	}
  \end{figure}  
\\
  When $t\gg 0$, we showed that
  $$\phi_{(s(\varepsilon), t)}(\vv') > \phi_{(s(\varepsilon), t)}(\vv)> \phi_{(s(\varepsilon), t)}(\uu).$$
  When $t_{2}<t<t_{1}$, we have
  $$\phi_{(s(\varepsilon), t)}(\vv)> \phi_{s(\varepsilon), t}(\vv')> \phi_{(s(\varepsilon), t)}(\uu),$$
  since $t_{3}<t_{2}$. When $t_{3}< t < t_{2}$, since $(s(\varepsilon), t)$ is below $W(\vv, \uu)$, we have $\phi_{(s(\varepsilon), t)}(\uu)> \phi_{(s(\varepsilon), t)}(\vv)$. Since $(s(\varepsilon), t)$ is above $W(\uu, \vv')$, we have $\phi_{(s(\varepsilon), t)}(\vv')> \phi_{(s(\varepsilon), t)}(\uu)$. However, since $(s(\varepsilon), t)$ is below $W(\vv, \vv')$, we have $\phi_{(s(\varepsilon), t)}(\vv)> \phi_{(s(\varepsilon), t)}(\vv')$. Hence when $t_{3}< t < t_{2}$, we have
  $$\phi_{(s(\varepsilon), t)}(\uu)> \phi_{(s(\varepsilon), t)}(\vv)>\phi_{(s(\varepsilon), t)}(\vv')> \phi_{(s(\varepsilon), t)}(\uu),$$
  a contradiction. In the following we assume $k\neq 0$. 
\\
  \textbf{Step 2: Show $k>0, e>0$.} We first show $ke>0$. Suppose not, we compute the $H_{\varepsilon}$-slopes of $\vv', \uu, \vv$. We have $\mu_{H_{\varepsilon}}(\vv')=\frac{dH^{2}}{r}+\frac{2e\varepsilon}{kr}$, $\mu_{H_{\varepsilon}}(\uu)=\frac{dH^{2}}{r} $, and $\mu_{H_{\varepsilon}}(\vv)=\frac{dH^{2}}{r}+2\varepsilon $. Let $s_{0}=\frac{dH\cdot H_{\varepsilon}}{rH_{\varepsilon}^{2}}$. Then for $t\gg 0$, we have
  $$\phi_{(s_{0}, t)}(\vv')<\phi_{(s_{0}, t)}(\uu) <\phi_{(s_{0}, t)}(\vv).$$
  Let $(s_{0}, t_{0})$ be the intersection point of $W(\vv', \vv)$ with the vertical line $\{s=s_{0}\}$. Since $W(\vv', \vv)$ is above $W_{-1}$ on $\mathbf{b}$, we have $t_{0}\geq \sqrt{\frac{2a}{rH_{\varepsilon}^{2}}-(\frac{dH^{2}}{rH_{\varepsilon}^{2}} )^{2}}$. Since $\uu, \vv, \vv'$ are independent, we have  $t_{0}> \sqrt{\frac{2a}{rH_{\varepsilon}^{2}}-(\frac{dH^{2}}{rH_{\varepsilon}^{2}} )^{2}}$. At $\sigma_{0}=\sigma_{(s_{0}, t_{0})}$, we have $\phi_{\sigma_{0}}(\vv)= \phi_{\sigma_{0}}(\vv')$. However, since every numerical wall of $\uu$ must pass through $\sigma_{\uu}$ (Definition \ref{blackhole}) and $\sigma_{0}$ is above $\sigma_{\uu}$, we have
  $$\phi_{\sigma_{0}}(\vv')< \phi_{\sigma_{0}}(\uu) < \phi_{\sigma_{0}}(\vv). $$
  This is a contradiction. Hence $ke>0$. 

  Since $\vv'$ is $\mathcal{A}$-effective, we have
  $$0\leq H_{\varepsilon}((kdH+eL)-krs(\varepsilon')H_{\varepsilon})=2e\varepsilon + \varepsilon'krH_{\varepsilon}^{2}.$$
  Hence $k, e$ cannot both be negative, we must have $k>0, e>0$.
\\
  \textbf{Step 3: The possible values of $k$ are $\frac{1}{2}, 1, \frac{3}{2}, 2$. }Since $\vv-\vv'$ is $\mathcal{A}$-effective and $k>0 $, we have
  $$0\leq H_{\varepsilon}((1-k)dH+(r-e)L - (1-k)rs(\varepsilon')H_{\varepsilon})=2\varepsilon (r-e) +(1-k)r\varepsilon' H_{\varepsilon}^{2}< 2\varepsilon (r-e) + r\varepsilon' H_{\varepsilon}^{2}. $$
  Since $0< \varepsilon' \ll \varepsilon \ll 1$ and $r$ is independent of the choice of $\varepsilon$ and $\varepsilon'$, we have $0< e\leq r$. Since $(\vv')^{2}=-2$, we have
  $$m=ka-\frac{k^{2}+e^{2}-1}{kr}.$$
  Since $\vv\vv'<0$ and $r\geq 2$, we have
\begin{equation}\label{long}
  0>\vv \vv'=-k-2er+\frac{e^{2}-1}{k}+kr^{2} > kr^{2}-k-2r^{2}-\frac{1}{k} \geq 4(k-2)-k-\frac{1}{k}.
\end{equation}
If $k\geq \frac{5}{2} $, then $0>\vv \vv'>-\frac{9}{10} $. Since $\vv \vv'\in \mathbb{Z}$, $\vv \vv'=0$, a contradiction.
  Hence the only possibilities for $k$ are $k=\frac{1}{2}, 1, \frac{3}{2}, 2$. 
\\
\textbf{Step 4: Consider all possibilities of $k$.}
\\
  \textit{When $k=1$:} Since $\vv\vv'<0$, we have
$$  0>\vv\vv'= -k-2er+\frac{e^{2}-1}{k}+kr^{2} =-1-2er+e^{2}-1+r^{2}=(e-r)^{2}-2. $$
Hence $|e-r|\leq 1$. If $|e-r|=1$, then $\mathrm{gcd}(e,r)=1$. Note that
  $$m=ka-\frac{k^{2}+e^{2}-1}{kr}=a-\frac{e^{2}}{r}\in \mathbb{Z},$$
  we must have $r=1$. This was excluded by assumption. Hence $e=r$ and $\vv'=\vv$, we have $W=W_{-1}$.
\\
   \textit{When $k=2$:} Since all inequalities hold in (\ref{long}), we must have $2er=2r^{2}$, hence $e=r$. Since
  $$m=ka-\frac{k^{2}+e^{2}-1}{kr}=2a-\frac{r^{2}+3}{2r}\in \mathbb{Z}, $$
  we have $r|3$. By assumption $r>1$, hence $r=e=3$. However, we have
  $$\vv\vv'=-k-2er+\frac{e^{2}-1}{k}+kr^{2}= -2-2\cdot 9 +\frac{9-1}{2}+2\cdot 9=2>0. $$
  This contradicts our assumption that $\vv\vv'<0$.
\\
   \textit{When $k=\frac{3}{2}$:} We have
  $$0>\vv\vv'=-\frac{13}{6}+\left(\sqrt{\frac{3}{2}}r-\sqrt{\frac{2}{3}}e\right)^{2}\geq -\frac{13}{6}.  $$
  Since $\vv\vv'\in \mathbb{Z}$, we have $\vv\vv'=-1$ or $-2$. In the former case, we have $3r-2e=\pm \sqrt{7}$, which is impossible. In the latter case, we have $3r-2e=\pm 1$. Since $0<e\leq r$, we must have $3e\leq 3r=2e+1$, $e\leq 1$. Since $k_{1}=3$ is odd, $e=\frac{1}{2}$. Then $3r=2e+1=2$, which is impossible. 
\\
  \textit{When $k=\frac{1}{2}$:} We have
  $$0>\vv\vv'=-\frac{5}{2}+(\sqrt{2}e-\frac{r}{\sqrt{2}})^{2}\geq -\frac{5}{2}. $$
  Since $\vv\vv'\in \mathbb{Z}$, we have $(\sqrt{2}e-\frac{r}{\sqrt{2}})^{2}$ is $-\frac{1}{2}$ or $-\frac{3}{2}$. In the latter case, we have $2e=r\pm \sqrt{3}$, which is impossible. In the only case left, we have $e_{1}=r \pm 1$. Then
  $$m=ka-\frac{k^{2}+e^{2}-1}{kr}=\frac{a}{2}-\frac{e_{1}^{2}-3}{2r}=\frac{a}{2}-\frac{1}{2}\cdot \frac{e_{1}^{2}-3}{r}. $$
  Since $a$ is odd, we have
  $$\frac{e_{1}^{2}-3}{r}=\frac{(r\pm 1)^{2}-3}{r}=\frac{r^{2}\pm 2r -2}{r}=(r\pm 2) -\frac{2}{r}  $$
  is an odd integer. Hence $r=2$ and $e_{1}=1$ or $3$. 

  Hence if either $\mathrm{Pic}(X)=\mathrm{Cl}(X)$ or $r\neq 2$, then there is no wall on $\mathbf{b}$ which is above $W$, $\widetilde{E}=F_{+}$. By Proposition \ref{equivalent}, $\widetilde{E}$ descends to a vector bundle $E$ on $X$. The $H$-stability of $E$ follows from $\widetilde{H}$-stability of $\widetilde{E}$.
\end{proof}

\begin{remark}\label{mod8}
Note that in the last case of the proof of Theorem \ref{existence}, since $\frac{dH+e_{1}L}{2}$ is a divisor on $\widetilde{X}$, we have that
  $$\left( \frac{dH+e_{1}L}{2} \right)^{2}=\frac{d^{2}H^{2}-2e_{1}^{2}}{4}=d^{2}\frac{H^{2}}{4}-\frac{e_{1}^{2}}{2}   $$
  is an even integer. Hence we must have $H^{2}\equiv 2 (\mathrm{mod}~8)$. 
\end{remark}

\begin{corollary}\label{existenceanduniqueness}
Let $\vv=(r, dH, a)\in  \mathrm{H}^{*}_{alg}(X)$ be a spherical Mukai vector with $r>0$ and either condition in Theorem \ref{existence} is satisfied. Then $M_{X, H}(\vv)$ is a reduced point. Furthermore, the unique sheaf in $M_{X, H}(\vv)$ is locally free.
\end{corollary}

\begin{proof}
  By Theorem \ref{existence}, $M_{X, H}(\vv)$ is non-empty. Since $v^{2}=-2$, $d^{2}H^{2}/2-ra=-1$, hence $\mathrm{gcd}(r,d)=1$. If $E_{1}, E_{2}\in M_{X, H}(\vv)$, then $E_{1}, E_{2}$ are $H$-stable. We have $\chi(E_{1}, E_{2})=-\vv^{2}=2$, either $\mathrm{Hom}(E_{1}, E_{2})$ or $\mathrm{Hom}(E_{2},E_{1})\cong \mathrm{Ext}^{2}(E_{1},E_{2})^{*}$ is nonzero. By stability, any nonzero map between $E_{1}$ and $E_{2}$ is an isomorphism. Hence the unique sheaf in $M_{X, H}(\vv)$ is constructed in Theorem \ref{existence}. It is locally free.

\end{proof}

Next we observe that all stable spherical sheaves on a nodal $K3$ surface are constructed by Theorem \ref{existence}.

\begin{lemma}\label{necessary}
  Using the notation in Theorem \ref{existence}, let $E$ be an $H$-stable spherical sheaf on $X$. Then $\pi^{*}(E)$ is an $H_{\varepsilon}$-stable spherical bundle for $0<\varepsilon \ll 1$. 
\end{lemma}

\begin{proof}
  Suppose $F\subset \pi^{*}(E)$ is an $H_{\varepsilon}$-destablizing subsheaf of $\pi^{*}(E)$ for sufficiently small $\varepsilon$. Then $\mu_{H}(F)>\mu(\pi^{*}(E))$. Since the mapping cone of $\pi^{*}\pi_{*}(F) \rightarrow F$ is supported on $L$ and $HL=0$, we have
  $$\mu_{X, H}(\pi_{*}(F))=\mu_{\widetilde{X}, H}(F)>\mu_{\widetilde{X}, H}(\pi^{*}(E))=\mu_{X, H}(E). $$
  Hence $\pi_{*}(F)\subset E$ is an $H$-destablizing subsheaf of $E$. 
\end{proof}

Next we show that the conditions in Theorem \ref{existence} are optimal.

\begin{theorem}[Nonexistence]\label{nonexistence}
Let $X$ be a general nodal $K3$ surface with $\mathrm{Pic}(X)=\mathbb{Z}H$. If $\mathrm{Cl}(X)\neq \mathrm{Pic}(X)$, then there exists no rank two $H$-stable spherical sheaf on $X$.
\end{theorem}

\begin{proof}
  All cohomology computations are taken on $\widetilde{X}$ unless otherwise stated.
  Let $D\in \mathrm{Pic}(\widetilde{X})$ be a divisor that is not generated by $H$ and $L$. Since the class group of an ordinary double point is $\mathbb{Z}/ 2 \mathbb{Z}$, we have $2D'=aH+bL$ for some $a,b\in \mathbb{Z}$. Since
  $$a^{2}H^{2}-2b^{2}=(2D')^{2}\equiv 0 (\mathrm{mod}~8) $$
 and $H^{2}$ is even, $a,b$ must have the same parity. They are both odd. (Note that we must have $H^{2}\equiv 2 (\mathrm{mod}~8)$.)

 Let $\uu=(2, dH, a)\in  \mathrm{H}^{*}_{alg}(X)$ be a spherical Mukai vector. Since $d$ is odd and $D'\in \mathrm{Pic}(\widetilde{X})$, we have a divisor $D=\frac{dH-L}{2}\in \mathrm{Pic}(\widetilde{X})$.

 Suppose there exists $E\in M_{\widetilde{X}, H_{\varepsilon}}(\uu)$, by Lemma \ref{necessary}, $\pi^{*}(E)\in M_{\widetilde{X}, H_{\varepsilon}}(u)$ for $0< \varepsilon \ll 1$. Let $M=\pi^{*}(E)\otimes \mathcal{O}_{\widetilde{X}}(-D)\in M_{\widetilde{X}, H_{\varepsilon}}((2, L, 0))$. Next we explicitly construct $M$.

 Consider the sheaf $N$ that fits into the following exact sequence:
 \begin{equation}\label{defN}
   0 \longrightarrow N \longrightarrow \mathcal{O}_{\widetilde{X}}^{\oplus 2} \overset{ev}{\longrightarrow} \mathcal{O}_{L}(1) \longrightarrow 0.
   \end{equation}
   We show that $N$ is $H_{\varepsilon}$-stable spherical vector bundle. Suppose $N_{1}$ is the first $H_{\varepsilon}$-JordanH\"{o}lder factor of $N$. If $N$ is not stable, then $N_{1}$ is a rank 1 sheaf. By Mukai's Lemma, $N_{1}$ is spherical, hence a line bundle. Write $N_{1}=\mathcal{O}_{\widetilde{X}}(a'H+b'L)$ for some $a', b' \in \frac{1}{2}\mathbb{Z} $. Since $N_{1}\subset \mathcal{O}_{\widetilde{X}}^{\oplus 2}$ and $\mu_{H_{\varepsilon}}(N_{1})\geq \mu_{H_{\varepsilon}}(N)=-\varepsilon$, we have $a'=0$ and $b'\in \mathbb{Z}$. Since $c_{1}(N)=L\in \mathrm{Pic}(\widetilde{X})$ is primitive, we have  $$2b'\varepsilon=\mu_{H_{\varepsilon}}(\mathcal{O}_{\widetilde{X}}(b' L))> \mu_{H_{\varepsilon}}(N)=-\varepsilon.$$ 
   Hence $b'\geq 0$. Since $\mathcal{O}_{\widetilde{X}}(b'L)$ is a subsheaf of $\mathcal{O}_{\widetilde{X}}^{\oplus 2}$, $b'=0$ and $N_{1}=\mathcal{O}_{\widetilde{X}}$. Since $N$ is the kernel of an evaluation map, we have $ \mathrm{H}^{0}(\widetilde{X}, N)=0$, a contradiction. Hence $N$ is $H_{\varepsilon}$-stable. Since $N$ is simple rigid, it is locally free. Taking the dual of (\ref{defN}), we get
   $$0 \longrightarrow \mathcal{O}_{\widetilde{X}}^{\oplus 2} \longrightarrow N^{*} \longrightarrow \underline{Ext}_{\widetilde{X}}^{1}(\mathcal{O}_{L}(1), \mathcal{O}_{\widetilde{X}})\cong \mathcal{O}_{L}(-3) \longrightarrow 0.$$
   Since $N\in M_{\widetilde{X}, H_{\varepsilon}}((2, -L, 0))$, we have $N^{*}\in M_{\widetilde{X}, H_{\varepsilon}}((2, L, 0))$. Hence $M\cong N^{*}$, we have an exact sequence
\begin{equation}\label{defM}
  0 \longrightarrow \mathcal{O}_{\widetilde{X}}^{\oplus 2} \longrightarrow M \longrightarrow \mathcal{O}_{L}(-3) \longrightarrow 0.
  \end{equation}
  Restricting the sequence (\ref{defM}) to $L$, we have
  $$0 \longrightarrow \underline{Tor}^{\widetilde{X}}_{1}(\mathcal{O}_{L}(-3), \mathcal{O}_{L})\cong \mathcal{O}_{L}(-1) \overset{coev}{\longrightarrow} \mathcal{O}_{L}^{\oplus 2} \longrightarrow M|_{L} \longrightarrow \mathcal{O}_{L}(-3) \longrightarrow 0. $$
  Hence $M|_{L}$ fits into the following exact sequence
  $$0 \longrightarrow \mathcal{O}_{L}(1) \longrightarrow M|_{L} \longrightarrow \mathcal{O}_{L}(-3) \longrightarrow 0. $$
  Since $\mathrm{Ext}^{1}_{L}(\mathcal{O}_{L}(-3), \mathcal{O}_{L}(1))=0$, we have $M|_{L}\cong \mathcal{O}_{L}(1)\oplus \mathcal{O}_{L}(-3)$. Twisting by $D$, we see that $\pi^{*}(E)|_{L}\cong \mathcal{O}_{L}(2)\oplus \mathcal{O}_{L}(-2)$.

  Consider the sheaf $F$ on $\widetilde{X}$ that fits into the following exact sequence
  $$0 \longrightarrow F \longrightarrow \pi^{*}(E) \longrightarrow \mathcal{O}_{L}(-2) \longrightarrow 0.$$
  Since $\pi_{*}(\mathcal{O}_{L}(-2))=0$, we have $\pi_{*}(F)\cong \pi_{*}\pi^{*}(E)\cong E$. Hence we have a non-zero map $f: \pi^{*}(E) \rightarrow F$ induced by the map
  $\mathrm{id}: E \rightarrow \pi_{*}(F)$ on $X$. Note that $\pi^{*}(E)$ is stable spherical sheaf, hence $\mathrm{Hom}(\pi^{*}(E), \pi^{*}(E))=\mathbb{C}$. In the following exact sequence
  $$0 \longrightarrow \mathrm{Hom}(\pi^{*}(E),F) \longrightarrow \mathrm{Hom}(\pi^{*}(E),\pi^{*}(E)) \overset{res}{\longrightarrow} \mathrm{Hom}(\pi^{*}(E), \mathcal{O}_{L}(-2)), $$
the map $res$ is nonzero. Hence $\mathrm{Hom}(\pi^{*}(E), F)=0$, a contradiction. 
 
\end{proof}

\begin{example}[Nodal $K3$ surfaces with $\mathrm{Pic}(X)\cong \ZZ$ and $\Cl(X)\neq \Pic(X)$]\label{contraction}
	This example shows that the nodal $K3$ surfaces in Theorem \ref{nonexistence} exist. Let $\XX \subset \PP^{3}$ be a very general quartic surface that contains a line $L$. Then $\Pic(X)=\ZZ h \oplus \ZZ L$, where $h$ is the hyperplane class of $\PP^{3}$. We have 
	$$ h^{2}=4,~ h\cdot L= 1,~ L^{2}=-2. $$
	
	First note that the cone of nef divisors $\mathrm{Nef}^{1}(\XX)_{\RR}=\langle 2h+L, h-L \rangle$. The class $h-L$ is effective, since it is the residual curve of $L$ in a plane containing $L$. Note that $(h-L)^{2}=0$, hence $h-L$ is base point free, otherwise $(h-L)^{2}$ would be positive. Hence $h-L$ generates an extremal ray of $\mathrm{Nef}(\XX)_{\RR}$. To see $2h+L$ is nef, note that it has no base point away from $L$, hence it suffices to check it has no base point on $L$. Let $x\in L$ be a point. Since $h-L$ is base point free, we may choose a plane $P$ which contains $L$, such that the residual curve of $L$ in $P$ misses $x$. Denote this residual curve by $C$. Since $C$ is a cubic curve in $P$, $C=P\cap S$ for a cubic surface $S$. The residual curve of $C$ in $S$ is a member of $|3h-(h-L)|=|2h+L|$ that misses $x$. Hence $2h+L$ is base point free, in particular it is nef. Since $L$ is effective and $(2h+L)\cdot L=0$, $2h+L$ generates an extremal ray of $\mathrm{Nef}^{1}(\XX)_{\RR}$. 
	
	Let $\pi: \XX \rightarrow X$ be the morphism defined by $H:= 2h+L$. Then $\pi$ contracts the $(-2)$-curve $L$, its image $X$ has a node. We have 
	$$ \HH^{1}(X, \OO_{X})=\HH^{1}(\XX, \OO_{\XX})=0, \omega_{X}=\pi_{*}(\pi^{*}\omega_{X})=\pi_{*}\OO_{\XX}=\OO_{X}.$$
	Hence $X$ is a nodal $K3$ surface. We have $\Cl(X)=\Pic(\XX)/(L)=\ZZ h$, and 
	$$ \Pic(X)=L^{\perp }= \ZZ (2h+L)=\ZZ H \subset \Pic(\XX). $$
	Hence $\Cl(X)\neq \Pic(X)$, and we have $H=2h$ on $X$. On $\XX$, we have $h=\frac{H-L}{2}$.

\end{example}

\begin{example}[Very general singular quartics]\label{P3}
  This example is an application of Corollary \ref{existenceanduniqueness}. We say a quartic surface $X\subset \mathbb{P}^{3}$ is a \emph{very general singular quartic surface}, if $[X]\in \mathbb{P} \mathrm{H}^{0}(\mathbb{P}^{3}, \mathcal{O}_{\mathbb{P}^{3}}(4H))$ is a very general point in the discriminant hypersurface. We claim that a very general quartic surface $X$ is a general nodal $K3$ surface in the sense of Definition \ref{nodalK3} (see e.g. \cite{EH16}).

  Note that $H^{2}=4\equiv 4 (\mathrm{mod}~8)$, by Remark \ref{mod8} and Corollary \ref{existenceanduniqueness}, there is exactly one stable spherical vector bundle on such $X$ for every spherical Mukai vector. Hence for a general smooth quartic surface $X'$ and a spherical Mukai vector $\vv\in \mathrm{H}^{*}_{alg}(X')$, acquiring a nodal singularity is not an obstruction for $\vv$ to lift to an exceptional bundle on $\mathbb{P}^{3}$. 

To see the claim, let $\mathcal{Q}$ be the period domain in the moduli space of weight 2 Hodge structures on $\mathbb{Z}^{22}$. Then $\mathcal{Q}$ is 20-dimensional \cite{Huy16}. The closure of the locus
  $$S=\{V\in \mathcal{Q}: V^{1,1}\cap \mathbb{Z}^{22}\cong \mathbb{Z}H \oplus \mathbb{Z}L, H^{2}=4, L^{2}=-2, HL=0\}\subset \mathcal{Q}$$
  has codimension 2, hence $\mathrm{dim}\overline{S}=18$. By the Torelli Theorem, we get an 18-dimensional family $\mathcal{S}$ of $K3$ surfaces, whose very general member $\widetilde{Y}$ has
  $$\mathrm{Pic}(\widetilde{Y})=\mathbb{Z}H \oplus \mathbb{Z}L, H^{2}=4, L^{2}=-2, HL=0.$$
  On $\widetilde{Y}\in \mathcal{S}$, let $\pi$ be the morphism defined by $H$. Since $ \mathrm{H}^{0}(\widetilde{Y}, \mathcal{O}_{\widetilde{X}}(H))=\mathbb{C}^{4}$, the image of $\pi$ is a surface in $\mathbb{P}^{3}$ with degree $d$. Let $X$ be a general singular quartic surface whose resolution of singularity is $\widetilde{X}$, then $\widetilde{X}$ is a member of $\mathcal{S}$. Hence $d=4$. In this way we get an 18-dimensional family of non-isomorphic singular quartic surfaces in $\mathbb{P}^{3}$. Acting by $\mathbb{P}\mathrm{GL}_{4}$, we get a family of singular quartic surfaces. The dimension of this family is $18+15=33$. Since the discriminant hypersurface in $\mathbb{P} \mathrm{H}^{0}(\mathbb{P}^{3}, \mathcal{O}_{\mathbb{P}^{3}}(4H))$ has dimension 33, a very general singular quartic $X \in \Gamma$ is a very general member of $\mathcal{S}$, hence is a general nodal $K3$ surface. Note that in this case $\mathrm{Cl}(X)= \mathrm{Pic}(X)$. 
  
\end{example}

\section{References}

\bibliographystyle{alpha}
\renewcommand{\section}[2]{} 
\bibliography{sphvbsing}

\end{document}